\documentclass[11pt]{article}
\usepackage{enumitem}
\usepackage{rotating}
\usepackage[english]{babel}
\usepackage[utf8]{inputenc}
\usepackage{amsthm}
\usepackage{amsmath}
\usepackage{amssymb}
\usepackage{amsfonts}
\usepackage{xypic}
\usepackage{bbm}
\usepackage{ifthen}
\usepackage{mathdots}
\usepackage{stmaryrd} % for \trianglelefteqslant
\usepackage{calc}
\usepackage{xcolor}
\usepackage{mathrsfs}

\title{...}

\author{}
\newtheorem{thm}{Theorem}[section]
\newtheorem{Def}[thm]{Definition}
\newtheorem{ex}[thm]{Example}

\newcommand{\N}{\mathbb{N}}
\newcommand{\Z}{\mathbb{Z}}
\newcommand{\C}{\mathbb{C}}
\newcommand{\Q}{\mathbb{Q}}
\newcommand{\Br}{\operatorname{Br}}
\newcommand{\cha}{\operatorname{char}}
\newcommand{\Frac}{\operatorname{Frac}}
\newcommand{\Gal}{\operatorname{Gal}}
\newcommand{\ind}{\operatorname{ind}}
\newcommand{\Ind}{\operatorname{Ind}}
\newcommand{\lcm}{\operatorname{lcm}}
\newcommand{\Mat}{\operatorname{Mat}}
\newcommand{\GL}{\operatorname{GL}}
\newcommand{\per}{\operatorname{per}}
\newcommand{\tensorF}{\operatorname{\otimes_F}}

\newtheorem{lemma}[thm]{Lemma}

\begin{document}
\title{Embedding problems of division algebras}
\author{Annette Maier}
\date{\today}
\maketitle
\begin{abstract} A finite group $G$ is called admissible over a given field if there exists a central division algebra that contains a $G$-Galois field extension as a maximal subfield. We give a definition of embedding problems of division algebras that extends both the notion of embedding problems of fields as in classical Galois theory, and the question which finite groups are admissible over a field. In a recent work by Harbater, Hartmann and Krashen, all admissible groups over function fields of curves over complete discretely valued fields with algebraically closed residue field of characteristic zero have been characterized. We show that also certain embedding problems of division algebras over such a field can be solved for admissible groups.  
\end{abstract}

\section{Introduction}

A finite group $G$ is called \textit{admissible} over a field $F$ if there exists a central division algebra over $F$ that contains a $G$-Galois field extension as a maximal subfield. Equivalently, $G$ is admissible over $F$ if and only if there exists a $G$-crossed product algebra over $F$ that is a division algebra. The question, which groups are admissible over a given field $F$ is a variant of the inverse Galois problem and was first posed by Schacher in 1968 (\cite{Schacher}).\\
\\
Schacher proved that for $F=\Q$, any admissible group has metacyclic Sylow subgroups (\cite[Thm. 4.1]{Schacher}) and he conjectured that the converse also holds, i.e., that any group with metacyclic Sylow subgroups is admissible over $\Q$. The structure of such groups was described in \cite{ChiSo}. Schacher also proved that for any finite group $G$ there exists an algebraic number field $F$ such that $G$ is admissible over $F$. On the contrary, if $F$ is a global field of characteristic $p>0$ and $G$ is admissible over $F$, then every Sylow subgroup of $G$ except possibly the $p$-Sylow subgroup is metacyclic (\cite[Thm. 10.3]{Schacher}). The problem whether every finite group with metacyclic Sylow subgroups is admissible over $\Q$, was solved for solvable groups by Sonn in \cite{Sonn} but is in general still open. However, many groups with metacyclic Sylow subgroups have been shown to be admissible over $\Q$ and for most of them, even all algebraic number fields over which these groups are admissible have been described explicitly (see for example \cite{FV87}, \cite{FF90}, \cite{SS92}, \cite{Fei93}, \cite{Lid96}, \cite{LZ96}, \cite{Fei02}, \cite{Fei04}). There has also been some work on admissibility over $\Q(t)$ and $\Q((t))$, see for example \cite{FS95a} and \cite{FS95b}. \\
\\
For $F=\C((t))(x)$ and more generally for $F$ a function field in one variable over a complete discretely valued field with algebraically closed residue field, the following theorem was proven in \cite{HHK}:

\begin{thm}[Harbater-Hartmann-Krashen] \label{HHK}  
Let $F$ be a finitely generated field extension of transcendence degree one over a complete discretely valued field $K$ with algebraically closed residue field $k$ and let $G$ be a finite group of order not divisible by $\cha(k)$. Then $G$ is admissible over $F$ if and only if each of its Sylow subgroups is abelian and metacyclic, that is, abelian of rank at most $2$.
\end{thm}

The forward direction was proven in a similar way as in Schacher's proof over $\Q$. The additional condition that the Sylow subgroups are abelian is due to the fact that $F$ contains all roots of unity of order prime to $\cha(k)$. The converse direction of Theorem \ref{HHK} was proven by using the method of patching over fields which was introduced in \cite{HH}. For each Sylow subgroup $P$, the authors construct a division algebra $D_P$ with a maximal subfield $E_P$ that is a $P$-Galois field extension (these are constructed over some field extensions of $F$) and then they show that these local building blocks can be patched into a division algebra over $F$ with a maximal subfield that is a $G$-Galois extension of $F$. We note that as in the case of algebraic number fields, not every $G$-Galois extension of $F$ has to be a maximal subfield of a division algebra if $G$ is admissible. For example $C_2\times C_2$ is admissible over $F$ by Theorem \ref{HHK}, but not every $C_2\times C_2$ extension of $F$ is a maximal subfield of a division algebra (see Example \ref{ex2}). In \cite{NefPar}, the authors use a similar strategy as in \cite{HHK} to prove that Theorem~\ref{HHK} is also true if $F$ is the fraction field of a complete local domain of dimension $2$ with separably closed residue field, e.g., $F=\C((X,Y))$. Patching over fields is also used in \cite{RS12}, where the authors give necessary conditions for a group to be admissible over function fields of curves over complete discretely valued fields with arbitrary residue field, e.g., $F=\Q_p(x)$. \\
\\
In the present paper, we introduce the notion of an embedding problem of division algebras and study embedding problems over fields $F$ as in Theorem~\ref{HHK}. In classical Galois theory, an embedding problem over a field $F$ is given by an exact sequence \[1 \to N \to G \to H \to 1 \] together with an $H$-Galois extension $E/F$. The embedding problem is called split if the exact sequence splits. A $G$-Galois extension $\tilde E/F$ with an \mbox{$H$-equivariant} embedding $E \subseteq \tilde E$ is called a proper solution to the embedding problem. Asking whether all embedding problems over $F$ have proper solutions is a stronger version of the inverse Galois problem over $F$ and is related to the structure of the absolute Galois group of $F$ by a theorem of Iwasawa, which asserts that if the absolute Galois group of $F$ is of countably infinite rank then it is free if and only if every embedding problem over $F$ has a proper solution (see \cite[Cor. 24.8.3]{FriedJarden}). For $F$ a field as in Theorem~\ref{HHK}, a result of Harbater and Pop (see Theorem~5.1.4 and Theorem~5.1.9 in the survey article \cite{Harb}) implies that every finite split embedding problem over $F$ has a proper solution. Now to define what an embedding problem of division algebras over a field $F$ is, we assume that we are given an exact sequence as above together with an $H$-Galois extension $E/F$ that is contained as a maximal subfield in a central division algebra $D$ over $F$. The concept should extend the notion of an embedding problem of fields, so a solution should consist of a pair $(\tilde D, \tilde E)$ with $\tilde D$ a central division algebra over $F$ and $\tilde E$ a maximal subfield of $D$ such that $\tilde E$ solves the embedding problem on the level of fields and such that $\tilde D$ relates to $D$ in a yet to be determined way. In our definition, we require that the $|N|$-th power of the Brauer class of $\tilde D$ equals the Brauer class of $D$. Asking whether such an embedding problem has a proper solution is stronger than asking whether $G$ is admissible. We go back to a field $F$ as in Theorem~\ref{HHK} and ask whether any embedding problem of division algebras over $F$ can be solved for an admissible group $G$ of order not divisible by $\cha(k)$. In Theorem~\ref{thmcoprime}, we give an affirmative answer in the case that $N$ and $H$ have coprime orders. The proof uses patching over fields and mimicks the construction in \cite{HHK}. For each Sylow subgroup $P$ of $N$, we use the building blocks $(D_P, E_P)$ constructed in \cite{HHK} and show that these can be patched together with the given pair $(D,E)$ in a way that yields a solution $(\tilde D, \tilde E)$ of the embedding problem of division algebras. Unfortunately, the method of patching does not seem to be applicable to the case where $N$ and $H$ have non-coprime orders, as will be explained in Section~\ref{disc}. \\
\\
The paper is organized as follows. In Section \ref{secebp}, we motivate our definition of an embedding problem of division algebras. Section \ref{secpat} provides some background on the method of patching over fields with a focus on patching central simple algebras. In Section \ref{seccoprime}, we prove that split embedding problems for admissible groups over fields $F$ as in Theorem \ref{HHK} can be solved under the assumption that the orders of kernel and cokernel in the exact sequence are coprime and not divisible by $\cha(k)$ and that further the given Galois extension is not ramified all along the closed fibre of a suitable regular model of $F$. Regarding the case where kernel and cokernel do not have coprime orders, we show that there exist proper solutions to the particular embedding problem given by an exact sequence of the form $1 \rightarrow C_n \to C_m\times C_n \to C_m \rightarrow 1$ with $n$ dividing some power of~$m$, in Section \ref{ncex}. In Section \ref{secdis}, we first discuss the use of patching for the non-coprime case and then conclude with giving an example of a solution of an embedding problem of fields that cannot be extended to a solution of an embedding problem of division algebras. 

\noindent \textbf{Acknowledgments.} I would like to thank David Harbater and Julia Hartmann for various helpful comments and discussions. 

\section{Embedding problems for division algebras}\label{secebp}
We first fix some notation. If $F$ is a field and $A$ is a central simple algebra, we denote its class in the Brauer group $\Br(F)$ by $[A]$. We write the group structure in $\Br(F)$ multiplicatively, i.e. $[A\otimes_F B]=[A]\cdot [B]$. All $F$-division algebras are assumed to be central over $F$.\\
\\
We now state our definition of an embedding problem of division algebras and then proceed with giving some reasons for this definition in Lemma \ref{converse}. The definition was suggested in this form by Harbater, Hartmann and Krashen. 
\begin{Def}\label{defebp}
Let $F$ be a field. A \textbf{(split) embedding problem $\mathcal{E}$ of division algebras} consists of a (split) exact sequence
$$1 \rightarrow N \rightarrow G \rightarrow ^{\! \! \! \! \!   f }  G/N \rightarrow 1$$ of finite groups together with a division algebra $D$ over $F$ such that $D$ contains a maximal subfield $E$ that is a Galois extension of $F$ with $\Gal(E/F) \cong G/N$. \\
A \textbf{proper solution} to $\mathcal{E}$ is a division algebra $\tilde D$ over $F$ that contains a maximal subfield $\tilde E$ which is a Galois extension of $F$ with $\Gal(\tilde E/F) \cong G$ such that the following holds: 
\begin{enumerate}
  \item $\tilde{E}$ is a proper solution to the embedding problem on the level of fields, i.e. $\tilde E \supseteq E$ and the diagram $$\xymatrix{\Gal(\tilde E/F) \ar[r]^{\operatorname{res}} \ar@{-}[d]_{\cong} & \Gal(E/F) \ar@{-}[d]^{\cong}\\
            G \ar[r]_{f} & G/N}$$ commutes.
 \item $[\tilde D]^{|N|}=[D]$ inside $\Br(F)$. 
\end{enumerate}
\end{Def}

\noindent The following lemma justifies Definition \ref{defebp} in so far that it provides a converse under the assumption that period and index are equal. 
\begin{lemma}\label{converse} Let $G$ be a finite group and let $F$ be a field such that $\per(D)=\ind(D)$ holds for all division algebras of index dividing $|G|$. Suppose we are given a division algebra $\tilde D$ over $F$ that contains a $G$-Galois extension $\tilde  E/F$ as a maximal subfield. For a normal subgroup $N \trianglelefteq G$, we then let $D$ be the division algebra in the Brauer class of $[\tilde D]^{|N|}$. Then $D$ contains the field of invariants $E:=\tilde E^N$ as a maximal subfield.  
\end{lemma}
\begin{proof}
 We have $\ind(D)=\per(D)=\per(\tilde{D})/|N|=|G/N|=[E:F]$, so it suffices to show that $E$ is a splitting field of $D$ (see \cite[Cor. 13.3]{Pie}). As $\tilde D$ contains $E$ as a maximal subfield, it is a $G$-crossed product algebra over $F$. Hence $\tilde D \otimes_F E$ is Brauer equivalent to an $N$-crossed product algebra over $E$ (see Lemma a. in Section 14.7 of \cite{Pie}). Hence $\ind(\tilde D\otimes_F E)$ divides $|N|$, the degree of this crossed product. On the other hand, $|G|=\ind(\tilde D)$ divides $[E:F]\ind(\tilde D\otimes_F E)$ (see \cite[Prop.\,13.4\,(v)]{Pie}), so $|N|$ equals $\ind(\tilde D\otimes_F E)$. We conclude that $[D\otimes_F E]=[\tilde D \otimes_F E]^{|N|}$ is trivial in $\Br(F)$, so $E$ splits $D$. 
\end{proof}

The following theorem asserts that the assumption $\per(D)=\ind(D)$ holds in the situation we are interested in, i.e., over fields $F$ as in the following theorem and groups $G$ such that $\cha(k)$ doesn't divide $|G|$. 

\begin{thm}[{\cite[Thm. 6.3]{ML} or \cite[Thm. 5.5]{HHK1}}] \label{per=ind} 
Let $F$ be a field of transcendence degree one over a complete, discretely valued field $K$ with algebraically closed residue field $k$. Then $\per(\alpha)=\ind(\alpha)$ holds for elements in $\Br(F)$ of period not divisible by $\cha(k)$. 
\end{thm}

\section{Patching}\label{secpat}
We briefly introduce the method of patching over fields established in \cite{HH} with emphasis on patching central simple algebras. We refer the reader to Section 5 and 6 of \cite{HH} and to Section 4 of \cite{HHK} for more details on this particular patching setup.\\
\\
Let $T$ be a complete discrete valuation ring with uniformizer $t$, fraction field $K$ and algebraically closed residue field $k$. Let further $F$ be an algebraic function field over $K$. Then there exists a regular connected projective $T$-curve $\hat X$ with function field $F$ such that its closed fibre $X$ has regular irreducible components. Let $S \subset X$ be any nonempty finite set that contains all points at which distinct irreducible components of $X$ meet. By a branch of $X$ at a point $Q \in S$ we mean a pair $\wp=(U,Q)$ with $U$ a connected component of $X\backslash S$ such that $Q$ lies in its closure $\overline{U}$. We now define the set $\Xi$ as the collection of all points $Q \in S$, all connected components $U$ of $X\backslash S$ and all branches $\wp$ of $X$ at points $Q \in S$. We endow $\Xi$ with a partial order by setting $U\succ \wp$ and $Q \succ \wp$ for any branch $\wp=(U,Q)$ in $\Xi$. \\
\\
Following Section 6 of \cite{HH}, we define fields $F_\xi$ for each $\xi \in \Xi$. If $\xi=Q \in S$, then we let $R_Q$ be the local ring of $\hat X$ at $Q$, we write $\hat R_Q$ for its completion at the maximal ideal, and we let $F_Q$ be the fraction field of $\hat R_Q$. If $\xi=U$ is a connected component of $X\backslash S$, then we let $R_U\leq F$ be the ring of rational functions that are regular on $U$, we let $\hat R_U$ be its $t$-adic completion, and we let $F_U$ denote the fraction field of $\hat R_U$. A branch $\wp=(U,Q)$ corresponds to a height-one prime ideal of $\hat R_Q$ that contains $t$ and localizing $\hat R_Q$ at $\wp$ yields a discrete valuation ring $R_\wp$. We let $\hat R_\wp$ be the completion of $R_\wp$ and we let $F_\wp$ be its fraction field. Then for each branch $\wp=(U,Q)$ in $\Xi$, we have $F_Q, F_U \subseteq F_\wp$ and $F=F_Q \cap F_U$. In particular, $F_{\xi_1} \subseteq F_{\xi_2}$ holds if $\xi_1 \succ \xi_2$. \\
\\
A patching problem for the system $\{F_\xi \ | \ \xi \in \Xi \}$ is a system $\{ V_\xi \ | \ \xi \in \Xi \}$ of $F_\xi$-vector spaces $V_\xi$ such that for each $\xi_1 \succ \xi_2$, we have an $F_{\xi_1}$-linear map $\phi_{\xi_1,\xi_2}: V_{\xi_1} \rightarrow V_{\xi_2}$ that induces an isomorphism  $V_{\xi_1}\otimes_{F_{\xi_1}} F_{\xi_2} \rightarrow V_{\xi_2}$. By Theorem 6.4. in \cite{HH}, every patching problem over $\{F_\xi \ | \ \xi \in \Xi \}$ has a solution, that is, there exists a vector space $V$ over $F$ with $F_\xi$-isomorphisms $V\otimes_F {F_\xi} \cong V_\xi$ for all $\xi \in \Xi$ that are compatible with the maps $\phi_{\xi_1,\xi_2}$ (for all $\xi_1 \succ \xi_2$). This solution $V$ can be identified with the inverse limit $\varprojlim \limits_{\xi \in \Xi} V_{\xi}$ (\cite[Prop. 2.3]{HH}).\\
\\
By Theorem 7.2. in \cite{HH}, there also exist solutions to patching problems of central simple algebras and $G$-Galois algebras over $\{F_\xi \ | \ \xi \in \Xi \}$, where $G$ denotes a finite group. We now describe the common approach to construct a $G$-Galois field extension of $F$ via patching (we will work out the details in the proof of Theorem \ref{thmcoprime}): We choose subgroups $H_1,\dots H_r$ of $G$ that generate $G$, construct $H_i$-field extensions $L_i$ over $F_{Q_i}$ (for suitable points $Q_1,\dots,Q_r \in S$) that split over $F_\wp$ for each branch $\wp$ at $Q_i$ and let $E_{Q_i}=\Ind^G_{H_i}(L_i)$ be the induced $G$-Galois algebra over $F_{Q_i}$. (We refer the reader to Chapter 4.3 of \cite{JLY} for details on induced Galois algebras.) For each component $U$ of $X\backslash S$, we let $E_U$ be the split $G$-Galois algebra $E_U=F_U^{\oplus|G|}=\Ind^G_{1}(F_U)$ (i.e. $G$ acts via permuting the components) and similarly $E_\wp=\Ind^G_{1}(F_\wp)$ for each branch $\wp$ in $\Xi$. These $G$-Galois algebras define a patching problem and we denote the solution by $E$. The building blocks $L_i$ must be chosen in such a way that $E$ is in fact a field. Similarly, in order to construct a division algebra over $F$, we construct central simple algebras over $F_\xi$, $\xi \in \Xi$, that define a patching problem. In order to show that the solution is in fact a division algebra and not just a central simple algebra, the following theorem is useful:
\begin{thm}[{\cite[Thm. 5.1]{HHK1}}]\label{lcm} 
If $A$ is a central simple algebra over $F$, then its index is the least common multiple of $\ind(A\otimes_F F_Q)$ (all $Q \in S$) and all $\ind(A\otimes_F F_U)$, where $U$ ranges over the connected components of $X\backslash S$.
\end{thm}
If we would like to show that the division algebra $A$ we constructed in this way is also a solution to a given embedding problem, we moreover have to show that a certain power of $A$ is Brauer equivalent to a given division algebra. The following theorem asserts that it is sufficient to show that this holds locally. 

\begin{thm}[{\cite[Thm. 7.2]{HH}}]\label{Braueriso}  
The base change map $$\Br(F)\to \prod\Br(F_U) \times_{\prod \Br(F_\wp)} \prod \Br(F_Q)$$ from the Brauer group of $F$ to the fibre product of the Brauer groups of $\Br(F_U)$ ($U$ ranges over the components of $X\backslash S$) and $\Br(F_Q)$ ($Q \in S$) with respect to the maps $\Br(F_U)\rightarrow \Br(F_\wp)$ and $\Br(F_Q) \rightarrow \Br(F_\wp)$ for all branches $\wp=(U,Q)$ is a group isomorphism. 
\end{thm}
This result is stated in \cite[Thm 7.2.]{HH} only under the stronger assumption that there exists a smooth model for $\hat X$. However, the proof of surjectivity only relies on the existence of solutions to patching problems, so the proof also works under the assumptions of this section. Using a result on simultaneous factorization of matrices that was later given in \cite{HHK1}, the proof of injectivity as given in \cite[Thm 7.2.]{HH} can also be transferred to our setup. We sketch the proof: Let $D$ be a division algebra over $F$ of degree $n$ that splits over all fields $F_Q$ and over all fields $F_U$. We have to show that $D$ is trivial, i.e., $D=F$. By Theorem 7.1. in \cite{HH}, it is sufficient to show that there are isomorphisms $\psi_Q:\Mat_n(F_Q)\to D\otimes_F F_Q$ for all $Q \in S$ and isomorphisms $\psi_U:\Mat_n(F_U)\to D\otimes_F F_U$ for all components $U$ of $X \backslash S$ such that for each branch $\wp=(U,Q)$ we have $$\psi_Q \otimes_{F_Q} F_\wp=\psi_U\otimes_{F_U}F_\wp \colon \Mat_n(F_\wp) \to D \tensorF F_\wp.$$ We start with arbitrary isomorphisms $\tilde \psi_Q:\Mat_n(F_Q)\to D\otimes_F F_Q$ (all $Q \in S$) and $\tilde \psi_U:\Mat_n(F_U)\to D\otimes_F F_U$ (all components $U$ of $X \backslash S$) and consider for each branch $\wp=(U,Q)$ the isomorphism $$(\tilde \psi_Q^{-1}\otimes_{F_Q}F_\wp) \circ (\tilde \psi_U \otimes_{F_U}F_\wp): \Mat_n(F_\wp)\to \Mat_n(F_\wp).$$ This has to be an inner automorphism, so it is given by conjugating with a matrix $C_{\wp} \in \GL_n(F_\wp)$. Now Theorem 3.6. in \cite{HHK1} implies that the matrices $C_\wp$ (where $\wp$ ranges over all branches) can be simultaneously factorized as follows. There exists matrices $C_Q$ (for each $Q \in S$) and $C_U~\in~\GL_n(F_U)$ (for each component $U$ of $X \backslash S$) such that for each branch $\wp=(U,Q)$, we have $C_\wp=C_QC_U$ with respect to the natural inclusions $F_Q \subseteq F_\wp$ and $F_U \subseteq F_\wp$. Defining $\psi_Q: \Mat_n(F_Q) \to D \tensorF F_Q$ as the conjugation with $C_Q$ followed by $\tilde \psi_Q$ and defining $\psi_U: \Mat_n(F_U) \to D \tensorF F_U$ as the conjugation with $C_U^{-1}$ followed by $\tilde \psi_U$ yields the desired result.

\section{The coprime case}\label{seccoprime}
In this section, we prove that embedding problems of division algebras can be solved for admissible groups over algebraic function fields $F/K$ with $K$ a complete discretely valued field with algebraically closed residue field $k$ and valuation ring $T\subset K$ under the following assumptions:
The orders of the kernel and cokernel in the exact sequence of the embedding problem are coprime and not divisible by $\cha(k)$ and for some regular connected projective $T$-curve $\hat X$ with function field $F$ and regular irreducible components of the closed fibre the following ramification assumption holds: Let $\hat Y$ denote the normalization of $\hat X$ inside the given Galois extension $E/F$. Then the cover $\hat Y \to \hat X$ is not ramified all along the closed fibre $X$ of $\hat X$. 
\begin{thm}\label{thmcoprime}
Let $F$ be a finitely generated field extension of transcendence degree one over a complete discretely valued field $K$ with algebraically closed residue field $k$. Let $\mathcal{E}$ be an embedding problem of division algebras such that the orders of the normal subgroup $N$ and of the factor group $G/N$ are coprime and not divisible by $\cha(k)$ and such that every Sylow subgroup of $G$ is abelian of rank at most $2$. Assume further that there exists a regular model $\hat X$ for $F$ such that the given $(G/N)$-extension $E/F$ is not ramified all along the closed fibre of $\hat X$. Then there exists a proper solution to $\mathcal{E}$.
\end{thm}

\begin{proof} 
Let $\mathcal{E}$ be such an embedding problem of division algebras, i.e., we are given an exact sequence of finite groups $1 \rightarrow N \rightarrow G \rightarrow G/N \rightarrow 1$ together with a $G/N$-Galois field extension $E/F$ that is a maximal subfield of a division algebra $D$ over $F$. Note that since $|N|$ and $|G/N|$ are coprime, the exact sequence $1 \rightarrow N \rightarrow G \rightarrow G/N \rightarrow 1$ splits by the theorem of Schur-Zassenhaus. We fix a subgroup $H\leq G$ that is mapped isomorphically to $G/N$. \\ 
\\
We choose a regular model $\hat X$ for $F$ with closed fibre $X$ such that $E/F$ is not ramified all along $X$. Suppose that $|N|$ has $r$ distinct prime divisors $p_1, \dots, p_r$. We can choose distinct closed points $Q_1,\dots, Q_r \in X$ at which $X$ is regular and which are unramified in $E/F$. By Hensel's Lemma, $E$ then splits at $F_{Q_i}$, i.e., $E\tensorF F_{Q_i}\cong F_{Q_i}^{|H|}\cong \Ind^H_1(F_{Q_i})$. We let $S'$ be the finite set of points at which distinct irreducible components of $X$ meet and we set $S=S' \cup \{Q_1,\dots,Q_r\}$ (note that this is a disjoint union). We then define $\Xi$ and $F_\xi$ ($\xi \in \Xi$) as in Section \ref{secpat} and abbreviate $F_i:=F_{Q_i}$ for $1 \leq i \leq r$. \\
\\
For each $1\leq i \leq r$, we fix a $p_i$-Sylow subgroup $P_i$ of $N$. By assumption, each $P_i$ is abelian of rank at most $2$. Following the proof of Proposition 4.4 in \cite{HHK}, we choose a division algebra $D_i$ over $F_i$ for each $i\leq r$ that contains a maximal subfield $L_i$ which is a $P_i$-Galois extension of $F_i$ such that $L_i\otimes_{F_i} F_{\wp_i} \cong \Ind^{P_i}_1(F_{\wp_i})$ for each $i\leq r$, where $\wp_i$ denotes the (unique) branch of $X$ at $Q_i$.\\
For each $\xi \in \Xi$ we define a $G$-Galois algebra $\tilde E_\xi$ over $F_\xi$:
\begin{eqnarray*}
\xi=Q_i \ (1\leq i \leq r): & \tilde E_{Q_i}&:=\Ind^G_{P_i}(L_i) \\
\text{any other } \xi:  & \tilde{E}_\xi&:=\Ind^G_{H}(E\otimes_F F_\xi). 
\end{eqnarray*}
We set $h=|H|$, $n=|N|=[G\!:\! H]$ and $n_i=[G\!:\!P_i]$ for $i\leq r$. Since $n$ and $h$ are coprime, there is an integer $1 \leq r < h$ such that $rn\equiv 1 \mod h$. Let $D'$ be the division algebra in the Brauer class of $D^{\otimes r}$. Since $E$ is a maximal subfield of $D$ with $[E:F]=h$, we have $\ind(D)=h$ and thus $\ind(D')=\ind(D^{\otimes r})=\ind(D)=h$ (see \cite[Thm. 5.5c)]{Sal}). \\
For each $\xi \in \Xi$, we now define a central simple algebra $A_\xi$ of degree $|G|$ over $F_\xi$: \vspace{-0.3cm}
\begin{eqnarray*}
\xi=Q_i \ (1\leq i \leq r): & A_{Q_i}&:=\Mat_{n_i}(D_i) \\
\text{any other } \xi:  & A_\xi&:=\Mat_{n}(D'\otimes_F F_\xi).
\end{eqnarray*}
The proof now proceeds in five steps and works similarly to the proof of Lemma 4.2 in \cite{HHK}.\\
\\
\textit{Step 1:} We first embed $\tilde{E}_\xi$ into $A_\xi$ for all $\xi \in \Xi$. If $\xi=Q_i$ for some $i \leq r$, $\tilde{E}_\xi=\tilde E_{Q_i}$ is the direct sum of $n_i$ copies of $L_i$ which is a subfield of $D_i$, hence $\tilde{E}_\xi$ embeds diagonally into $\Mat_{n_i}(D_i)=A_\xi$. For any other $\xi \in \Xi$, note first that $E$ splits $D^{\otimes r}$, since $E$ is a maximal subfield of $D$ and thus splits $D$. As $D'$ is equivalent to $D^{\otimes r}$ in $\Br(F)$, $E$ is also a splitting field of $D'$, with $[E:F]=h=\deg(D')$. Thus $E$ is isomorphic to a maximal subfield of $D'$ (see \cite[Cor. 13.3]{Pie}). We conclude that $E\otimes_F F_\xi$ embeds into $D' \otimes_F F_\xi$ and $\tilde{E}_\xi$ thus embeds diagonally into $A_\xi$. We obtain $F_\xi$-algebra embeddings $\iota_\xi:\tilde{E}_\xi \hookrightarrow A_\xi$ for all $\xi \in \Xi$. \\
\\
\textit{Step 2:} We now use patching to obtain a commutative $G$-Galois algebra $\tilde E$. We need to show that $\tilde E_Q\otimes_{F_Q} F_\wp \cong \tilde E_\wp \cong\tilde E_U\otimes_{F_U} F_\wp$ for each branch $\wp=(U,Q)$ contained in $\Xi$. Let first $\wp=\wp_i=(U_i,Q_i)$ for some $i\leq r$ ($U_i$ denotes the unique component of $X$ such that $Q_i \in \overline{U_i}$). 
Then $$\tilde{E}_{Q_i} \otimes_{F_i}F_{\wp_i} \cong \Ind^{G}_{P_i}(L_i \otimes_{F_i}F_{\wp_i})\cong \Ind^{G}_{P_i}(\Ind^{P_i}_1(F_{\wp_i})) \cong {\Ind^G_1(F_{\wp_i})}.$$ On the other hand, we assumed that $F_i$ splits $E$, hence so does $F_{\wp_i}\supset F_i$. We conclude $\tilde E_{\wp_i}\cong \Ind^G_{H}(E\otimes_F F_{\wp_i})\cong\Ind^G_{H}(\Ind^H_1(F_{\wp_i}))\cong {\Ind^G_1(F_{\wp_i})}$ and also
$$\tilde{E}_{U_i} \otimes_{F_{U_i}}F_{\wp_i} \cong \Ind^{G}_{H}(E \otimes_{F} F_{\wp_i}) \cong {\Ind^G_1(F_{\wp_i})}.$$ For any other branch $\wp=(U,Q)$, 
$$\tilde{E}_Q \otimes_{F_Q}F_\wp \cong \Ind^{G}_{H}(E \otimes_{F} F_\wp) \cong \tilde{E}_U \otimes_{F_U}F_\wp$$ holds by definition. Note that all isomorphisms are compatible with the action of $G$ and we thus obtain isomorphisms of $G$-Galois $F_\wp$-algebras \\ $\phi_\wp: \tilde{E}_Q \otimes_{F_Q}F_\wp \rightarrow  \tilde{E}_U \otimes_{F_U}F_\wp$ for all branches $\wp=(U,Q)$ contained in $\Xi$. Hence $\{\tilde E_\xi \ | \ \xi \in \Xi \}$ defines a patching problem of $G$-Galois algebras over $\{F_\xi \ | \ \xi \in \Xi \}$ and we can now apply Theorem 7.1 (together with Theorem 6.4) of $\cite{HH}$ to obtain a solution, i.e., a $G$-Galois $F$-algebra $\tilde E$ with isomorphisms $\tilde E \otimes _F F_\xi \cong \tilde{E}_\xi$ for all $\xi \in \Xi$ that are compatible with the isomorphisms $\phi_\wp$. Thus $\tilde E$ is a commutative $G$-Galois algebra of dimension $|G|$ over $F$.\\
\\
\textit{Step 3:} We use patching to obtain a central simple algebra $A$ with $\tilde E \leq  A$. We first show that $A_Q\otimes_{F_Q}F_\wp \cong A_\wp \cong A_U\otimes_{F_U} F_\wp $ holds for all branches $\wp=(U,Q)$ in $\Xi$. Let first $\wp=\wp_i=(U_i,Q_i)$ for some $i \leq r$. Note that $L_i$ embeds into $F_{\wp_i}$, since $F_{\wp_i}$ splits $L_i$. As $L_i$ is a splitting field of $D_i$, the same is true for $F_{\wp_i}$ and we conclude  
$$A_{Q_i} \otimes_{F_i} F_{\wp_i} \cong \Mat_{n_i}(D_i \otimes_{F_i} F_{\wp_i})\cong \Mat_{|G|}(F_{\wp_i}).$$ 
On the other hand, $E$ embeds into $F_{\wp_i}$ (here we use again the assumption that $F_i\subset F_{\wp_i}$ splits $E$) and it is a splitting field of $D'$, so $F_{\wp_i}$ also splits $D'$: $A_{\wp_i}\cong \Mat_{n}(D' \otimes_{F} F_{\wp_i})\cong \Mat_{|G|}(F_{\wp_i})$ and similarly
$$A_{U_i} \otimes_{F_{U_i}} F_{\wp_i} \cong \Mat_{n}(D' \otimes_{F} F_{\wp_i})\cong \Mat_{|G|}(F_{\wp_i}).$$ 
For any other branch $\wp=(U,Q)$, 
$$A_Q\otimes_{F_Q} F_\wp \cong \Mat_{n}(D' \otimes_{F} F_\wp) \cong A_U \otimes_{F_U} F_\wp $$ holds by definition. 
Thus there exist $F_\wp$-algebra isomorphisms \\ $\tilde \psi_\wp: A_Q \otimes_{F_Q}F_\wp \rightarrow  A_U \otimes_{F_U}F_\wp$ for each branch $\wp=(U,Q)$ contained in $\Xi$. The same argument as in the proof of Proposition 4.2 in \cite{HHK} (involving the Skolem-Noether theorem) yields $F_\wp$-algebra isomorphisms $\psi_\wp: A_Q \otimes_{F_Q}F_\wp \rightarrow  A_U \otimes_{F_U}F_\wp$ for all branches $\wp=(U,Q)$ such that the following diagram commutes (and is compatible with the embedding $\tilde E_\wp \to A_\wp$): 
$$\xymatrix{A_Q \otimes_{F_Q}F_\wp \ar[rr]^{\psi_\wp}& & A_U \otimes_{F_U}F_\wp\\
            \tilde{E}_Q \otimes_{F_Q}F_\wp \ar@{_{(}->}[u]^{\iota_Q \otimes_{F_Q}F_\wp} \ar[rr]_{\phi_\wp}& & \tilde{E}_U  \ar@{_{(}->}[u]_{\iota_U \otimes_{F_U}F_\wp} \otimes_{F_U}F_\wp}$$ By Theorem 7.1 of \cite{HH} and the commutativity of the above diagram, we may patch the algebras $A_\xi$ to a central simple $F$-algebra $A$ of degree $|G|$ containing $\tilde E$ such that $A \otimes _F F_\xi \cong A_\xi$ holds for all $\xi \in \Xi$ (and these isomorphisms are compatible with the homomorphisms $\psi_\wp$). \\
\\
\textit{Step 4:} We show that $A$ is in fact a division algebra, i.e. $\ind(A)=|G|$ holds. First, note that for all $i \leq r$ 
$$|P_i|=[L_i:F_i]=\deg(D_i)=\ind(A_{Q_i})=\ind(A \otimes_F F_i) \ | \ \ind(A)$$
and thus $|N|=\lcm(|P_i|, i\leq r)$ divides $\ind(A)$. As $|N|$ and $|H|$ are coprime, we have $|G|=|N|\cdot|H|=\lcm(|H|,|N|)$ and it suffices to show that $|H|=\ind(D')$ divides $\ind(A)$. Let $\mathcal{U}$ be the collection of irreducible components of $X\backslash S$. Then Theorem \ref{lcm} yields $$\ind(D')=\lcm(\ind(D' \tensorF F_Q), \ind(D' \tensorF F_U) \ | \ Q \in S, \ U \in \mathcal U).$$ For $i\leq r$, $E\subset F_i$ is a splitting field of $D'$, hence $\ind(D' \tensorF F_i)=1$ and we conclude 
\begin{eqnarray*}
 |H|=\ind(D')&=&\lcm(\ind(D' \tensorF F_Q), \ind(D' \tensorF F_U) \ | \ Q \in S', U \in \mathcal{U}) \\ 
         &=&\lcm(\ind(A_Q), \ind(A_U) \ | \ Q \in S', U \in \mathcal{U} ) \\
         &=&\lcm(\ind(A\tensorF F_Q), \ind(A \tensorF F_U) \ | \ Q \in S', U \in \mathcal{U}) \ | \ \ind(A).
\end{eqnarray*}
Thus $\tilde D:=A$ is a division algebra and it follows immediately that $\tilde E$ is a maximal subfield with Galois group $G$. \\
\\
It remains to show that the pair $(\tilde D, \tilde E)$ solves the given embedding problem. To see that $E$ embeds into $\tilde E$ which is an inverse limit over $\tilde E \tensorF F_\xi$ ($\xi \in \Xi$), it suffices to show that $E$ embeds into all $\tilde E \tensorF F_\xi$. If $\xi=Q_i$, we use the identity $\Ind^G_H(L)^N \cong\Ind^{G/N}_{H/N\cap H}(L^{N\cap H})$ for induced algebras (see Remark (2) of Chapter 4.3. in \cite{JLY}) and compute 
\begin{eqnarray*}E\otimes_F F_i&\cong&\Ind^H_1(F_i) \cong \Ind^{G/N}_1(L_i^{P_i})\cong \Ind^{G/N}_1((\Ind^N_{P_i}(L_i))^N) \\
 &\cong& (\Ind^{G}_N(\Ind^N_{P_i}(L_i))^N \cong \tilde E_{Q_i}^N \cong \tilde E^N \tensorF F_i \hookrightarrow \tilde E \tensorF F_i
\end{eqnarray*}
 where all maps are $H$-equivariant. Thus we can embed $E$ into $\tilde E \tensorF F_i$ via an $H$-equivariant map. 
For any other $\xi \in \Xi$, we can embed $E$ into $E\tensorF F_\xi$ which embeds diagonally into $\Ind^G_H(E \tensorF F_\xi) \cong \tilde E \tensorF F_\xi$. Both of these embeddings are $H$-equivariant.
Altogether we obtain an $H$-equivariant inclusion $E \leq \tilde E$. \\
We conclude that the following diagram commutes:
$$\xymatrix{\Gal(\tilde E/F) \ar[r] \ar@{-}[d]_{\cong} & \Gal(E/F) \ar@{-}[d]^{\cong}\\
            G \ar[r]_f & G/N}$$

\noindent \textit{Step 5 :} We finally have to show that $[\tilde D]^n=[D]$ holds in $\Br(F)$. We use Theorem \ref{Braueriso} and show that equality holds for the images of $[\tilde D]^n$ and $[D]$ in the fibre product. If $Q=Q_i$ for an $i \leq r$, we have already seen that $F_i$ contains $E$ and thus splits $D$. Thus $[D\tensorF F_i]$ is trivial and $[\tilde D \tensorF F_i]^n=[A_{Q_i}]^n=[D_i]^n$ is trivial too, since $\per(D_i) \ | \ \ind(D_i)=|P_i| \ | \ n$. Hence $[\tilde D \tensorF F_{Q_i}]^n=[D\tensorF F_{Q_i}]$ for all $1\leq i \leq r$. If $\xi$ is either a point contained in $S'$ or one of the components of $X\backslash S$, we have 
\begin{eqnarray*}
[\tilde D \tensorF F_\xi]^n&=&[A_\xi]^n=[D'\tensorF F_\xi]^n=[D^{\otimes r} \tensorF F_\xi]^n \\
&=&[D \tensorF F_\xi]^{rn}=[D \tensorF F_\xi],
\end{eqnarray*}
 where the last equality follows from the fact that the period of $D \tensorF F_\xi$ divides $\ind(D)=h$ and $rn \equiv 1 \mod h$.
\end{proof}

\section{A non-coprime example}\label{ncex}
In this section, we let $F$ again denote a finitely generated field extension of transcendence degree one over a complete discretely valued field $K$ with algebraically closed residue field $k$. Let further $n,m \in \N$ be coprime to $\cha(k)$ and such that every prime factor of $n$ also divides $m$. For example, we could choose $n$ and $m$ both prime powers with respect to the same prime.\\
Now let $C_n$ and $C_m$ be cyclic groups of order $n$ and $m$ and consider the split exact sequence
$$1 \rightarrow C_n \longrightarrow \! \! \! \! \! \! \! \!   ^{\iota_2} \ \ C_m\times C_n \longrightarrow \! \! \! \! \! \! \! \! \! \!  ^{\pi_1} \ \ C_m \rightarrow 1,$$ where $\iota_2$ and $\pi_1$ denote the canonical inclusion and projection. Assume we are given a division algebra $D$ over $F$ containing a maximal subfield $E$ with $\Gal(E/F) \cong C_m$. We can solve this particular embedding problem without using patching. 

\subsection{Writing $D$ as a symbol algebra}
First note that since $\cha(k)$ doesn't divide $nm$ and since $k$ is algebraically closed, $K\subseteq F$ contains a primitive $nm$-th root of unity $\zeta$, by Hensel's Lemma. We set $\zeta_m=\zeta^n \in F$, a primitive $m$-th root of unity. \\
As $\Gal(E/F)\cong C_m$, Kummer theory asserts that there exists an $a \in F$ such that $E=F(^m \! \! \! \sqrt a)$. Now $E$ is a splitting field of $D$ and so $D$ is Brauer equivalent to a symbol algebra $(a,b,m,F,\zeta_m)$, for some $b \in F^\times$ (see \cite[\S 11, Lemma 1]{Draxl}). Recall that the central simple algebra $(a,b,m,F,\zeta_m)$ is generated by two elements $X$ and $Y$ such that $X^m=a$, $Y^m=b$ and $XY=\zeta_mYX$. As both $D$ and $(a,b,m,F,\zeta_m)$ are of degree $m$, they are actually isomorphic. We denote the Brauer class of $(a,b,m,F,\zeta_m)$ by $[a,b,m,F,\zeta_m]$. 

% \begin{lemma}\label{norm}
% The order of $b$ in $F/N_{E/F}(E^\times)$ equals $m$.
% \end{lemma}
% \begin{proof}
% It is clear that the order is less or equal than $m$, since $b^m=N_{E/F}(b)$. Now for an $i \in \N$, we have $[D]^i=[a,b,m,F,\zeta_m]^i=[a,b^i,m,F,\zeta_m]$ (see equation (11) following Lemma $3$ in \cite[\S 11]{Draxl}). We have $\ind(D)=m$, so the period of $D$ divides $m$ and is thus not divisible by the characteristic of $k$.  Using Theorem \ref{per=ind}, we obtain $\per(D)=\ind(D)=m$. Hence $[D]^i=[a,b^i,m,F,\zeta_m]$ is non-trivial in $\Br(F)$ for all $1\leq i \leq m-1$. Now a symbol algebra $(\alpha,\beta, m, F, \zeta_m)$ is trivial in $\Br(F)$ if and only if $\beta$ is a norm of the field extension $F(^m \! \! \! \sqrt \alpha)/F$, by Corollary $4$ of \cite[\S 11]{Draxl}. We conclude that $b^i$ is not contained in $N_{E/F}(E^\times)$ for all $1 \leq i \leq m-1$ and thus $b$ is of order at least $m$ in $F/N_{E/F}(E^\times)$.
% \end{proof}

\subsection{Construction of $\tilde D$ such that $[\tilde D]^n=[D]$}
We set $\tilde D:=(a,b,nm,F,\zeta)$, which is a central-simple algebra of degree $nm$ over $F$ (see \cite[\S 11, Thm.1]{Draxl}). We can now use Lemma $6$ in \cite[\S 11]{Draxl} to obtain $$[\tilde D]^n=[a,b,nm,F,\zeta]^n=[a,b,m,F,\zeta^n]=[a,b,m,F,\zeta_m]=[D].$$ By Theorem \ref{per=ind}, $[D]$ has order $\ind(D)=m$ in $\Br(F)$. Taking into account that every prime factor of $n$ also divides $m$, we deduce that $[\tilde D]$ has order $nm$ in $\Br(F)$. This means that $\per(\tilde D)=nm$, hence $\ind(\tilde D)=nm=\deg(\tilde D)$ and so $\tilde D$ is a division algebra with $[\tilde D]^n=[D]$.\\
\\
It remains to show that $\tilde D$ contains a maximal subfield $\tilde E$ which is a solution to the corresponding embedding problem of fields. 

\subsection{Construction of a suitable maximal subfield $\tilde E \subseteq \tilde D$}
Let $\tilde X, \tilde Y$ be generators of $\tilde D$ such that $\tilde X^{nm}=a$, $\tilde Y^{nm}=b$ and $\tilde X \tilde Y=\zeta \tilde Y \tilde X$. Then $E=F(^m \! \! \! \sqrt a)\cong F(\tilde X^n)$ is a (commutative) subfield of $\tilde D$. We set $\tilde E:=F(\tilde X^n, \tilde Y^m)\supseteq E$. This is also a (commutative) subfield of $\tilde D$, since $\tilde X^n \cdot \tilde Y^m=(\zeta^n)^m\tilde Y^m \tilde X^n=\tilde Y^m \tilde X^n$. Hence $\tilde E\cong F(^m \! \! \! \sqrt a, ^n \! \! \! \sqrt b )$ and we have to show that $\tilde E$ is Galois with group $C_m \times C_n$ over $F$. It suffices to show $[\tilde E:F]=nm$. Now $\tilde E$ is a splitting field of $\tilde D$ (see \cite[\S 11, Lemma 9]{Draxl}), so $nm=\ind(\tilde D)$ divides $[\tilde E:F]$. Therefore, $\tilde E$ is Galois with Galois group $C_m \times C_n$ over $F$. This isomorphism is given by 
\begin{eqnarray*}
\Gal(\tilde E/F)=\Gal(F(^m \! \! \! \sqrt a, ^n \! \! \! \sqrt b )/F) &\longrightarrow & \Gal(F(^m \! \! \! \sqrt a)/F) \times \Gal(F(^n \! \! \! \sqrt b)/F) \\ 
\sigma &\mapsto& (\sigma|_{F(^m \! \! \! \sqrt a)}, \sigma|_{F(^n \! \! \! \sqrt b)})
\end{eqnarray*} and we conclude that the following diagram commutes:
$$\xymatrix{\Gal(\tilde E/F) \ar[r]^{\operatorname{res}} \ar@{-}[d]_{\cong} & \Gal(E/F) \ar@{-}[d]^{\cong}\\
            C_m\times C_n \ar[r]_{\pi_1} & C_m}$$
So $\tilde E$ is in fact a solution to the corresponding embedding problem of fields.

\section{Discussion}\label{secdis} 
In this section, we first explain the difficulties that arise when one wants to use patching to solve embedding problems of division algebras in the non-coprime case.\\ In the second part of this section, we give a negative answer to a related question concerning whether solutions of embedding problems of fields can be extended to solutions of embedding problems of division algebras.  
\subsection{The non-coprime case} \label{disc}
Let $F$ be as in Theorem \ref{thmcoprime}. Let $\mathcal{E}$ be a split embedding problem given by an exact sequence $1\rightarrow N \rightarrow G \rightarrow G/N \rightarrow 1$ (with $G$ admissible and $|G|$ not divisible by $\cha(k)$) together with a division algebra $D$ containing a $(G/N)$-Galois extension $E/F$ as a maximal subfield. If $|N|$ and $|G/N|$ are not coprime, the method of patching can in general not be used to construct solutions: Assume we have fixed a regular model $(\hat X, S)$ of $F$ with corresponding set $\Xi$ as in Section \ref{secpat}. We would like to construct central simple algebras $A_\xi$ over $F_\xi$ (for all $\xi \in \Xi$) of degree $|G|$ that can be patched together to a division algebra $A$ which contains a solution $\tilde E$ of the embedding problem on the level of fields such that $[A]^{|N|}=[D]$ holds. In particular, $[A_U]^{|N|}=[A\otimes_F F_U]^{|N|}=[D\otimes_F F_U]$ has to hold for each component $U$ of $X\backslash S$. By Corollary 5.11 in \cite{HHK1}, $\ind=\per$ holds for $A_U$ and $D\otimes_F F_U$. We conclude that $\ind(A_U)=\ind(D\tensorF F_U)\cdot\gcd(|N|, \ind(A_U))$, and is thus divided by $\ind(D\tensorF F_U)\cdot\gcd(|N|, \ind(D \tensorF F_U))$. The index $|G/N|$ of the given division algebra $D$ ``usually'' doesn't get significantly smaller when tensoring with $F_U$ (see Example \ref{exnoncoprime} below). Unless we are working in a very special situation with models $(\hat X,S)$ such that the index of $D\otimes _F F_U$ decreases significantly (i.e., most of the primes that divide $\ind(D)$ don't divide $\ind(D\otimes_F F_U)$) for all components $U$ of $X\backslash S$, we would have to construct building blocks $A_U$ over $F_U$ of index not much smaller than $|G/N|\cdot \gcd(|N|, |G/N|)$. That is, patching just leaves us with pretty much the same problem over $F_U$ instead of $F$. 

\begin{ex}\label{exnoncoprime}
Let now $F=\C((t))(x)$ and fix a primitive sixth root of unity $\zeta_6 \in \C$. Consider an exact sequence of the form 
$$1 \rightarrow S_3 \rightarrow S_3\rtimes C_6 \rightarrow C_6 \rightarrow 1$$ such that $G=S_3\rtimes C_6$ is admissible, i.e., $C_6$ acts on $S_3$ in a way that $G$ contains $C_2\times C_2$ and $C_3\times C_3$. Together with the symbol algebra $D=(x,t,6,F,\zeta_6)$ with maximal subfield $E=F(^6{\! \! }\sqrt x)$, this defines an embedding problem of division algebras over $F$. In fact, even $D\otimes_F \C(x)((t))$ is a division algebra, which can be seen using Example 2.7. and 4.4. in \cite{Wad} (with valuation $v: \C(x)((t))^{\times} \rightarrow \Z \times \Z$ with $\Z \times \Z$ ordered lexicographically and $v_1(f)$ the $t$-adic valuation of $f \in F$ and $v_2(f)$ the $x$-adic valuation of the lowest coefficient of $f$ in its $t$-adic expansion; $n:=6$, $a:=t$ and $b:=x$). \\
\\
Now set $\hat X=\mathbb{P}^1_{\C[[t]]}$ and let $S$ be any finite subset of $X=\mathbb{P}^1_\C$. Then $\Xi$ consists of the points in $S$, the one component $U:=X\backslash S$ of $X\backslash S$ and branches $(U,Q)$ for each $Q \in S$. Then $F_U \subset F_\emptyset=\C(x)((t))$, hence $\ind(D\otimes_F F_U)\geq \ind(D\otimes_F \C(x)((t)))=6$, as seen above. In order to use patching we would have to construct a central simple algebra $A_U$ over $F_U$ of degree $36$ with $\ind(D\otimes_F F_U)\cdot \gcd(6, \ind(D\otimes_F F_U))=6\cdot 6=36$ dividing $\ind(A_U)$, hence $\ind(A_U)=36=\deg(A_U)$. In other words, $A_U$ would have to be a solution of the embedding problem over $F_U$ given by $1 \rightarrow S_3 \rightarrow S_3\rtimes C_6 \rightarrow C_6 \rightarrow 1$ with given division algebra $D\otimes_F F_U$ and maximal subfield $E\otimes_F F_U$. Patching with respect to any model of the form $(\mathbb{P}^1_{\C[[t]]},S)$ thus doesn't allow to break up the problem into smaller pieces. 
\end{ex}

\subsection{A related question}
One might also ask the following: Given a split embedding problem of division algebras $\mathcal{E}$ together with a Galois extension $\tilde E/F$ that solves the embedding problem on the level of fields, does there exist a division algebra over $F$ containing $\tilde E$ and solving the embedding problem of division algebras? This turns out to be wrong, as the following example demonstrates. 
\begin{ex} \label{ex2}
Let $F=\C((t))(x)$ and $\tilde E=F(\sqrt{1+tx}, \sqrt{1+tx^{-1}})$. Then $\tilde E$ is a $C_2\times C_2$-extension of $F$, but it is not a maximal subfield of a division algebra. Indeed, assume there exists a division algebra $\tilde D$ of degree $4$ containing $\tilde E$. We set $\hat X=\mathbb{P}^1_{\C[[t]]}$ and we let $Q$ be the point in the closed fibre $X$ given by $x=t=0$. We further set $S=\{Q\}$ and $U=X\backslash S$. Theorem~\ref{lcm} yields 
$$4=\ind(\tilde D)=\lcm(\ind(\tilde D\otimes_F F_Q), \ind(\tilde D\tensorF F_U)),$$ so at least one of $\tilde D\tensorF F_Q$ and $\tilde D \tensorF F_U$ has index $4$ and is thus a division algebra. We conclude that at least one of the commutative algebras $\tilde E \tensorF F_Q$ and $\tilde E \tensorF F_U$ is a field. Now $(1+tx)$ is a square in $\C[x][[t]] \subset \C[[x,t]] \subset \Frac(\C[[x,t]])=F_Q$, so $\tilde E \tensorF F_Q$ is not a field. Similarly, $(1+tx^{-1})$ is a square in $\C[x^{-1}][[t]] \subset  \Frac(\C[x^{-1}][[t]])=F_U$, so $\tilde E\tensorF F_U$ is not a field, either, a contradiction. \\
\\
On the other hand, $\tilde E$ is a solution to the embedding problem of fields given by the exact sequence
$$1\to C_2 \to C_2\times C_2 \to C_2 \to 1$$ and given field $E=F(\sqrt{1+tx})$. Furthermore, $E$ is contained as a maximal subfield in the symbol algebra $(1+tx,x,2,F,-1)$ of degree $2$, which is a division algebra (by \cite[\S 11, Cor.4]{Draxl} it is sufficient to check that $x$ is not a norm of $E/F$). So $(D,E)$ together with the exact sequence define an embedding problem of division algebras that can be solved (see Section \ref{ncex}) but there exists no solution $(\tilde D, \tilde E)$ with $\tilde E$ the given solution of the embedding problem of fields. 
\end{ex}

\textbf{Author information:}\\ Annette Maier, Lehrstuhl f\"ur Mathematik (Algebra), RWTH Aachen University, 52056 Aachen. Email: annette.maier@mathA.rwth-aachen.de

\begin{thebibliography}{KMRT98}
\bibitem[CS81]{ChiSo}
D. Chillag and J. Sonn. 
\newblock {\em Sylow-metacyclic groups and ${\Q}$-admissibility.}
\newblock { Israel J. Math. 40. (1981), no. 3-4, 307–323.} 

\bibitem[Dra83]{Draxl}
P. K. Draxl.
\newblock {\em Skew Fields.}
\newblock Cambridge University Press, 1983.

\bibitem[Fei93]{Fei93}
W. Feit.
\newblock {\em The $K$-admissibility of $2A_6$ and $2A_7$.}
\newblock Israel J. Math. 82 (1993), no. 1-3, 141–156. 

\bibitem[Fei02]{Fei02}
W. Feit.
\newblock {\em $\rm SL(2,11)$ is $\Q$-admissible.}
\newblock  J. Algebra 257 (2002), no. 2, 244–248. 

\bibitem[Fei04]{Fei04}
W. Feit.
\newblock {\em ${\rm PSL}_2(11)$ is admissible for all number fields. }
\newblock  Algebra, arithmetic and geometry with applications (West Lafayette, IN, 2000),  295–299, Springer, Berlin, 2004. 

\bibitem[FF90]{FF90}
P. Feit and W. Feit.
\newblock \textit{The $K$-admissibility of ${\rm SL}(2,5)$.}
\newblock {Geom. Dedicata 36 (1990), no. 1, 1–13.} 

\bibitem[FJ08]{FriedJarden}
M. D. Fried and M. Jarden.
\newblock \textit{Field Arithmetic.}
\newblock {Third edition. Results in Mathematics and Related Areas, vol. 11. Springer, 2008.} 

\bibitem[FS95a]{FS95a}
W. Feit and M. Schacher.
\newblock \textit{${\Q}(t)$ and ${\Q}((t))$-admissibility of groups of odd order.}
\newblock {Proc. Amer. Math. Soc. 123 (1995), no. 6, 1639–1645. } 

\bibitem[FS95b]{FS95b}
W. Feit and M. Schacher.
\newblock \textit{Crossed products over algebraic function fields.}
\newblock { J. Algebra 171 (1995), no. 2, 531–540. }  

\bibitem[FV87]{FV87}
W. Feit and P. Vojta.
\newblock \textit{Examples of some ${\Q}$-admissible groups.}
\newblock {J. Number Theory 26 (1987), no. 2, 210–226.} 

\bibitem[GL96]{LZ96}
Z. Girnius and S. Liedahl.
\newblock \textit{$K$-admissibility of $S_{12},S_{13},S_{14},S_{15}$.}
\newblock {J. Algebra 185 (1996), no. 2, 357–373.}

\bibitem[Har03]{Harb}
D. Harbater.
\newblock \textit{Patching and Galois theory.}
\newblock {Galois groups and fundamental groups,  313–424,
Math. Sci. Res. Inst. Publ., 41, Cambridge Univ. Press, Cambridge, 2003.} 

\bibitem[HH10]{HH}
D. Harbater and J. Hartmann. 
\newblock \textit{Patching over Fields.}
\newblock {Israel J. Math. 176 (2010), 61-107.} 

\bibitem[HHK09]{HHK1}
D. Harbater, J. Hartmann and D. Krashen. 
\newblock \textit{Applications of Patching to Quadratic Forms and Central Simple Algebras.}
\newblock {Invent. Math. 178 (2009), No.2, 231-263.}

\bibitem[HHK11]{HHK}
D.  Harbater, J.  Hartmann and D.  Krashen. 
\newblock \textit{Patching Subfields of Division Algebras.}
\newblock {Trans. Amer. Math. Soc. 363 (2011), No. 6, 3335-3349.}

\bibitem[JLY02]{JLY}
C. U. Jensen, A. Ledet, N. Yui.
\newblock {\em Generic Polynomials.}
\newblock Cambridge University Press, 2002.

\bibitem[Lie96]{Lid96}
S. Liedahl.
\newblock \textit{$K$-admissibility of wreath products of cyclic $p$-groups.}
\newblock {J. Number Theory 60 (1996), no. 2, 211–232.}


\bibitem[Lie09]{ML}
M. Lieblich.
\newblock \textit{Period and index in the Brauer group of an arithmetic surface (with an appendix by Daniel Krashen).}
\newblock {Preprint; arXiv:math/0702240v4}

\bibitem[NP10]{NefPar}
D. Neftin and E. Paran.
\newblock {\em Patching and admissibility over two-dimensional complete local domains.}
\newblock  Algebra Number Theory 4 (2010), no. 6, 743–762. 


\bibitem[Pie82]{Pie}
R. S. Pierce.
\newblock {\em Associative Algebras.}
\newblock Graduate Texts in Mathematics 88, Springer, 1982.

\bibitem[RS12]{RS12}
B. Surendranath Reddy and V. Suresh
\newblock {\em Admissibility of groups over function fields of p-adic curves.}
\newblock Preprint; arXiv:1201.1938v1

\bibitem[Sal99]{Sal}
D. J. Saltman.
\newblock {\em Lectures on Division Algebras.}
\newblock CBMS 94, 1999.

\bibitem[Sch68]{Schacher}
M. Schacher.
\newblock {\em Subfields of division rings, I.}
\newblock J. Algebra 9, (1968), 451–477.

\bibitem[Son83]{Sonn}
J. Sonn.
\newblock {\em ${\Q}$-admissibility of solvable groups.}
\newblock  J. Algebra 84, (1983), no. 2, 411–419. 

\bibitem[SS92]{SS92}
M. Schacher and J. Sonn.
\newblock {\em $K$-admissibility of $A_6$ and $A_7$.}
\newblock  J. Algebra 145 (1992), no. 2, 333–338. 

\bibitem[Wad02]{Wad}
A. R. Wadsworth.
\newblock \textit{Valuation theory on finite dimensional division algebras,} in \textit{Valuation theory and its applications.}
\newblock {Fields Inst. Commun., vol.32, Amer. Math. Soc., Providence, RI, 2002, pp. 385-449. }




\end{thebibliography}
\end{document}